\documentclass[12pt, reqno]{amsart}
\usepackage{amsmath, amsthm, amscd, amsfonts, amssymb, graphicx, color}
\usepackage[bookmarksnumbered, colorlinks, plainpages]{hyperref}

\newtheorem{theorem}{Theorem}[section]
\newtheorem{lemma}[theorem]{Lemma}
\newtheorem{proposition}[theorem]{Proposition}
\newtheorem{corollary}[theorem]{Corollary}
\theoremstyle{definition}
\newtheorem{definition}[theorem]{Definition}
\newtheorem{example}[theorem]{Example}

\theoremstyle{remark}

\numberwithin{equation}{section}

\begin{document}
\setcounter{page}{1}

\title[order almost Dunford-Pettis operators]{order almost Dunford-Pettis operators on Banach lattices}

\author[Ardakani and Modarres Mosadegh]
{H. Ardakani$^1$ and S.M.S. Modarres Mosadegh$^{2^*}$}

\address {$^{1, 2}$ Department of Mathematics, University of Yazd, P. O. Box 89195-741, Yazd, Iran.}

\email{\textcolor[rgb]{0.00,0.00,0.84}{halimeh\_ardakani@yahoo.com}}
\email{\textcolor[rgb]{0.00,0.00,0.84}{smodarres@yazd.ac.ir}}



\date{Received: xxxxxx; Revised: yyyyyy; Accepted: zzzzzz.
\newline \indent $^{*}$ Corresponding author}

\begin{abstract}

By introducing the concepts of order almost Dunford-Pettis and almost weakly limited operators in Banach lattices, we give some properties of them related to some well known classes of operators, such as, order weakly compact, order Dunford-Pettis, weak and almost Dunford-Pettis and weakly limited operators. Then, we characterize Banach
lattices $E$ and $F$ on which each operator from $E$ into $F$ that is order almost Dunford-Pettis
and weak almost Dunford-Pettis is an almost weakly limited operator.

\end{abstract} \maketitle

\section{Introduction and preliminaries}
A subset $A$ of a Banach space $X$ is called limited (resp.
Dunford--Pettis (DP)), if every weak$^*$ null (resp. weak null) sequence $(x_n^*)$ in $X^*$  converges uniformly on $A$, that is, $$\lim_{n\rightarrow \infty} \sup_{a\in A}|\langle a,x_n^* \rangle|=0.$$ We know that every relatively compact subset of $X$ is limited and every limited set is DP, but the converse of these assertions, in general,
are false. The reader can be find some useful and additional properties of limited and DP sets in \cite{Limited} and \cite{Andrews}.\\
Recently in \cite{Almost limited set} and \cite{Almost DP set}, the class of almost limited sets and almost DP sets are introduced in Banach lattices. A subset $A$ of a Banach lattice $E$ is called almost limited (resp.
almost DP), if every disjoint weak$^*$ null (resp. disjoint weak null) sequence $(x_n^*)$ in $E^*$ converges uniformly on $A$.\\
It is clear that every almost limited set is almost DP, but the converse is false, in general. For instance, $B_c=[-1,1]$ is an almost DP set, but it is not almost limited.\\
Based on the concept of DP (resp. limited) sets, the class of order DP (resp. order limited)
operators is defined in \cite{ordered DP}, \cite{Dunford-Pettis sets} and \cite{ordered limited}. In fact, an operator $T$ from a Banach lattice $E$ into
a Banach space $X$ is said to be order DP (resp. order limited) if it carries each order bounded
subset of $E$ into a DP (resp. limited) set of $X$, i.e., if for each $x\in E^+$, the subset
$T ([-x, x])$ is DP (resp. limited) in $X$.\\
An operator $T: X ‎\rightarrow‎ E$ is
called limited, whenever $T(B_X)$ is an limited set. By \cite{almost limited set}, $T: X ‎\rightarrow‎ E$ is
called almost limited, whenever $T(B_X)$ is an almost limited set in $E$.\\
Following Lin in \cite{weakly limited}, An operator $T: X ‎\rightarrow‎ Y$ is
called weakly limited whenever $T(B_X)$ is a DP set in $Y$.\\
The aim of this paper is to introduce new classes of operators that we call
order almost DP and almost weakly limited operators and give some interesting applications of this class
of operators.
Also we will give some equivalent conditions for $T(A)$ to be a almost DP set, where $A$ is an almost DP (solid) subset of a Banach lattice $E$ and $T$ is an operator from $E$ to $X$.\\
It is evident that if $E$ is a Banach lattice, then its dual $E^*$, endowed with the dual norm and pointwise order, is also a Banach lattice. The norm $\|.\|$ of a Banach lattice $E$ is order continuous if for each generalized net $(x_{\alpha})$ such that $x_{\alpha}\downarrow0$ in $E$, $(x_{\alpha})$ converges to $0$ for the norm $\|.\|$, where the notation $x_{\alpha}\downarrow0$ means that the net
$(x_{\alpha})$ is decreasing, its infimum exists and inf$(x_{\alpha})=0$. A Banach lattice is said to be $\sigma$--Dedekind complete if for its countable subset that is bounded above has a supremum. A subset $A$ of $E$ is called solid if $|x|\leq |y| $ for some $y\in A$ implies that $x\in A$ and the solid hull of $A$ is the smallest solid set including $A$ and is exactly the set $Sol(A)=\lbrace y\in E:|y| \leq |x|, $ \ for \ some\ $x\in A\rbrace$.\\
Throughout this article, $X$ and $Y$ denote the arbitrary Banach spaces and $X^*$ refers to the dual of the Banach space $X$. Also $E$ and $F$ denote arbitrary Banach lattices and $E^+ =\lbrace x\in E : x\geq 0\rbrace$ refers to the positive cone of the Banach lattice $E$ and $B_E$ is the closed unit ball of $E$. If $a,b$ belong to a Banach lattice $E$ and $a\leq b$, the interval $[a,b]$ is the set of all $x\in E$ such that $a\leq x \leq b$. A subset of a Banach lattice is called order bounded if it is contained in an order interval. The lattice operations in $E$ are weakly sequentially continuous, if for every weakly null sequence $(x_n)$ in $E$, $|x_n|\rightarrow 0$ for $\sigma (E,E^*)$. The lattice operations in $E^*$ are weak sequentially continuous, if for every weak$^*$ null sequence $(f_n)$ in $E^*$, $|f_n|\rightarrow 0$ for $\sigma (E^*,E)$.\\ We refer the reader to \cite{Aliprantis} and \cite{Meyer} for unexplained terminologies on Banach lattice theory and positive operators.

\section{order almost Dunford-Pettis operators} 
In this section we will define  new classes of operators so called order almost DP and almost weakly limited operators and establish some additional properties of them related to some operators. \\
A bounded linear operator $T$ from a Banach space $X$ into a Banach space $Y$ is DP (resp. almost DP), if it carries weakly null (resp. disjoint weakly null) sequences in $X$ to norm null ones. See \cite{Aliprantis} and \cite{Sanchez almost DP operator}.\\
It is clear that $T$ is weakly limited if and only if $T^*$ is DP.

\begin{definition}\label{I} An operator $T$ from $E$ into $F$ is said to be order almost DP
if it carries each order bounded subset of $E$ into an almost DP set in $F$.
\end{definition}
\noindent Note that there exist operators which are order almost DP, but fail to be
order DP. Indeed, $Id_{\ell_{\infty}}: \ell_{\infty}‎\rightarrow‎ \ell_{\infty}$ is order almost DP (because $[-e,e]=B_{\ell_{\infty}}$ is almost DP set), but it is not order DP (because $[-e,e]=B_{\ell_{\infty}}$ is not DP set‌). It is clear that each order DP operator is order almost DP.\\
\noindent As in \cite{Aliprantis}, an operator $T$ from $X$ into $Y$ is said to be weak DP,
if it carries relatively weakly compact sets in $E$ to DP ones.\\
By a simple proof we can investigate that each weakly limited operator $T: E ‎\rightarrow‎ X$ is order DP and weak DP.

\begin{definition}\label{II} An operator $T$ from $X$ into $E$ is said to be almost weakly limited
whenever $T(B_X)$ is a almost DP set in $E$.
\end{definition}

\noindent It is clear that $T$ is almost weakly limited if and only if $T^*$ is almost DP.\\
An operator $T: X ‎\rightarrow‎ E$ is called weak and almost DP if $T$ carries each relatively weakly compact set in $X$ to an almost DP set in $E$, equivalently, for
every weakly null sequence $(x_n)‎\subset‎ X$, and every disjoint weak null sequence
$(f_n)‎\subset‎ E^*$ we have $f_n(Tx_n)‎\rightarrow 0$. See \cite{weak DP property}.\\
By a simple proof we can investigate that each almost weakly limited operator is order almost DP and weak almost DP.\\
\begin{example}\label{III} Every weakly limited operator is almost weakly limited, but the converse is false, in general. In fact, since the closed unit ball $L^1[0,1]$ is almost DP set, but it is not DP set, $Id_{ L^1[0,1]}: L^1[0,1] ‎\rightarrow‎ L^1[0,1]$ is  almost weakly limited, but it not weakly limited.\\
Also, the identity operator $Id_{\ell_{\infty}}: \ell_{\infty}‎\rightarrow‎ \ell_{\infty}$ is almost weakly limited operator and it is not weakly limited.
\end{example}

\noindent A Banach lattice $E$ has the Schur (resp. positive Schur) property, if every weakly null (resp. weakly null with positive terms)  sequence in $E$ is norm null, See \cite{Aliprantis} and \cite{positive Schur}.

\begin{corollary}\label{P} Every operator $T$ from a Banach space $X$ into a Banach lattice $E$ such that $E^*$ has the Schur (resp. positive Schur) property is weakly limited (resp. almost weakly limited).
\end{corollary}
 \begin{proof} We note that dual Banach lattice $E^*$ has the Schur (resp. positive Schur) property if and only if the closed unit ball $B_E$ is DP (resp. almost DP) set.
\end{proof}

\begin{proposition}\label{PP}In every Banach lattice we have the following assertions:
\begin{itemize}
\item [(a)] Each limited operator is weakly limited, but the converse is false, in general. Indeed $Id_{c_0}: c_0‎\rightarrow‎ c_0$ is weakly limited operator and it is not limited.
\item [(b)]  Each limited operator is almost limited, but the converse is false, in general. Indeed $Id_{\ell_{\infty}}: \ell_{\infty}‎\rightarrow‎ \ell_{\infty}$ is almost limited, but it is not limited.
\item [(c)]  Each almost limited operator is almost weakly limited, but the converse is false, in general. Indeed $Id_c: c‎\rightarrow‎ c$ is almost weakly limited, but it is not almost limited.
\end{itemize}
\end{proposition}

\noindent By \cite[Definition 3.6.1]{Meyer}, a subset $A$ in a Banach lattice $E$ is $L$-weakly compact, if every disjoint sequence in $Sol(A)$ is norm null. By \cite{Meyer} every $L$-weakly compact set is relatively weakly compact and the converse holds for Banach lattices with the positive Schur property, see \cite{positive Schur}. An operator $T$ from a Banach space $X$ into a Banach lattice $E$ is $L$-weakly compact, if $T(B_X)$ is an $L$-weakly compact set in $E$.

\begin{theorem}\label{PPP} Every $L$-weakly compact operator is almost weakly limited, but the converse is false, in general.
 \end{theorem}
\begin{proof}
By \cite[Proposition 2.8]{Almost DP set}, every $L$-weakly compact set in a Banach lattice is an almost DP set. So every every $L$-weakly compact operator almost weakly limited.\\
Also by \cite[Theorem 4.2]{Almost limited set}, every $L$-weakly compact operator on a Banach lattice is an almost limited operator and so is an almost weakly limited operator.\\
The converse is false. In fact, $Id_{c_0}: c_0‎\rightarrow‎ c_0$ is almost weakly limited operator and it is not $L$-weakly compact, because the closed unit ball $c_0$ is an almost DP set, but it is not an $L$-weakly compact set.
\end{proof}

\noindent  Remember that a Banach lattice $E$ is an $AL$-space if $x \wedge y=0$ in $E$ implies $\|x+y\|=\|x\|+\|y\|$.

\begin{theorem}\label{C}\cite{Almost DP set} Let $E$ be an $AL$-space, for a norm bounded subset $A$ of $E$, the following
statements are equivalent.
\begin{itemize}
\item [(a)] $A$ is $L$-weakly compact.
\item [(b)] $A$ is relatively weakly compact.
\item [(c)] $A$ is DP set.
\item [(d)] $A$ is almost DP set.
\end{itemize}
\end{theorem}

\noindent An operator $T$ from $E$ into $X$ is said to be an order weakly compact operator,
if it carries order intervals in $E$ to relatively weakly compact sets in $X$. By \cite{ordered almost limited}, an operator $T$ from $E$ into $F$ is said to be order almost limited
if it carries each order bounded subset of $E$ into an almost limited set of $F$.

\begin{theorem}\label{CC} Every order weakly compact operator from a Banach lattice $E$ into an $AL$-space $E$ is an order almost limited and order DP operator.
 \end{theorem}
\begin{proof}
For each $x\in E^+$, $T[-x,x]$ is relatively weakly compact set. By \cite{positive Schur}, every relatively weakly compact set in an $AL$-space is an $L$-weakly compact set and by \cite{Almost limited set} it is an almost limited set. Also by Theorem \ref{C}, every relatively weakly compact set in an $AL$-space is DP set.
\end{proof}

\noindent By \cite{AZIZ ELBOUR}, an operator $T$ from $X$ into $E$ is said to be weak and almost limited operator,
if it carries relatively weakly compact sets in $X$ to almost limited ones.

\begin{corollary}\label{CCC} Every $L$-weakly compact operator from a Banach space $X$ to a Banach lattice $E$ is a weak and almost limited operator.
 \end{corollary}
 \begin{proof}
We note that $T(B_X)$ is an $L$-weakly compact set.  By \cite{Almost limited set}, $T(B_X)$ is an almost limited set. So $T$ carries each relatively weakly compact sets in $X$ to almost limited ones.
\end{proof}

\noindent There is an order DP operator which is not $L$-weakly compact. In fact, $Id_{c_0}: c_0‎\rightarrow‎ c_0$ is an order DP operator and it is not $L$-weakly compact, because the closed unit ball $c_0$ is a DP set, but it is not an $L$-weakly compact set.

\begin{corollary}\label{V} Every $L$-weakly compact operator $T$ from a Banach lattice $E$ into an $AL$-space is an order DP operator.
 \end{corollary}
 
 \begin{proof} If $T$ is an $L$-weakly compact operator from a Banach lattice $E$ into an $AL$-space, $T(B_E)$ is an $L$-weakly compact set and so by Theorem \ref{C}, it is a DP set. Hence, $T$ is an order DP operator. 
\end{proof}

\noindent By \cite{Meyer}, an element $x$ belonging to a Riesz space $E$ is discrete, if $x>0$ and $|y| \leq x$ implies $y=tx$ for some real number $t$. If every order interval $[0,y]$ in $E$ contains a discrete element, then $E$ is said to be a discrete Riesz space.

\begin{theorem}\label{VV}\cite{almost limited set} Let $T$ be an operator from a Banach space into a discrete Banach lattice with order continuous norm. Then  the following
statements are equivalent.
\begin{itemize}
\item [(a)] $T$ is almost limited.
\item [(b)] $T$ is $L$-weakly compact.
\item [(c)] $T$ is limited.
\item [(d)] $T$ is compact.
\end{itemize}
\end{theorem}
 
 \noindent An operator $T$ from $E$ into $X$ is said to be an $AM$-compact operator,
if it carries order intervals in $E$ to relatively compact sets in $X$. It is clear that every $AM$-compact operator is an order DP and so order almost DP operator, but the converse is false, in general. For the converse see the following theorem.
 
 \begin{theorem}\label{VVV} Every order almost DP operator $T$ from a Banach lattice $E$ into a discrete $AL$-space $F$ is $AM$-compact operator and this condition on $F$ can not be removed.
 \end{theorem}
 \begin{proof} For each $x\in E^+$, $T[-x,x]$ is almost DP set and by Theorem \ref{C} it is $L$-weakly compact and also by Theorem \ref{VV} it is relatively compact. Hence $T$ is an $AM$-compact operator. \\
There is an order almost DP operator which is not $AM$-compact. Indeed $Id_{\ell_{\infty}}: \ell_{\infty}‎\rightarrow‎ \ell_{\infty}$ is order almost DP operator, but it is not $AM$-compact. Every order DP operator is not $AM$-compact, in general. Indeed $Id_c: c‎\rightarrow‎ c$ is order DP operator, but it is not $AM$-compact. 
\end{proof}

\begin{proposition}\label{M} We have the following statements.
\begin{itemize}
\item [(a)] Every almost limited operator from a Banach space $X$ into a discrete Banach lattice $E$ with order continuous norm is  limited.
\item [(b)] Every almost weakly limited operator from a Banach space $X$ into an $AL$-space $E$ is weakly limited.
\item [(c)] Every almost weakly limited operator from a Banach space $X$ into an $AL$-space $E$ is almost limited.
\item [(d)] Every weakly limited operator from a Banach space $X$ into a Grothendieck space is limited.
\end{itemize}
\end{proposition}

\begin{theorem}\label{MM} Every order bounded operator on a Banach lattice is an order almost DP operator.
\end{theorem}
\begin{proof} We note that every order bounded operator maps order intervals into order intervals. By \cite[Corollary 2.2]{Almost DP set}, each order interval of a Banach lattice is an almost DP set. 
\end{proof}

\noindent Remember that a Banach lattice $E$ is an $AM$-- space if $x \wedge y=0$ in $E$ implies $\|x\vee y\|=\max \lbrace\|x\|,\|y\|\rbrace$.

\begin{corollary}\label{MMM} If $E$ is a $AM$-space with unit, then identity operator $Id_E$ is almost weakly limited.
\end{corollary}
 \begin{proof} In each $AM$-space $E$ with unit, the closed unit ball $B_E$ is an order bounded set and so it is almost DP set. So, the identity operator $Id_E$ is almost weakly limited.
\end{proof}

\begin{theorem}\label{R} Every order almost limited operator from a Banach lattice $E$ into $F$ is an order almost DP operator, and the converse holds, if $F$ is an $AL$-space.
\end{theorem}
\begin{proof} Since each almost limited set in a Banach lattice is a almost DP set, so order almost limited operator is an order almost DP operator, but the converse is false, in general. Indeed, $Id_c: c‎\rightarrow ‎c$ is order almost DP (because $[-1,1]=B_c$ is almost DP set), but it is not order almost limited (because $[-1,1]=B_c$  is not almost limited set‌).\\
If $T$ is an order almost DP operator from $E$ to an $AL$-space $F$, then for each $x\in E^+$, $T[-x,x]$ is almost DP and by Theorem \ref{C}, it is $L$-weakly compact and so almost limited. Hence $T$ is an order almost limited set. 
\end{proof}

\begin{theorem}\label{RR} For an operator $T$ from a Banach lattice into an $AL$-space, the following are equivalent:
\begin{itemize}
 \item [(a)]  $T$ is order almost DP operator,
 \item [(b)]  $T$ is order DP operator,
 \item [(c)]  $T$ is order weakly compact operator,
 \item [(d)]  $T$ is order almost limited operator.
\end{itemize}
\end{theorem}
 
 \noindent It should be note that the assertions in Corollary \ref{RR} are not equivalent, in general.\\
 At first recall that by \cite{Weak DP Property Wnuk}, a Banach lattice $E$ has the weak DP property if every weakly compact operator on $E$ is an almost DP operator.\\
By \cite{Almost DP set}, a Banach lattice $E$ has the weak DP property if and only if every relatively weakly compact set in $E$ is an almost DP set.

 \begin{theorem}\label{RRR} Every order weakly compact operator $T$ from a Banach lattice into a Banach lattice $F$ with the weak DP property is order almost DP operator, and this condition on $F$ is essential.
 \end{theorem}
 
\begin{proof} Since every relatively weakly compact set in a Banach lattice $F$ with the weak DP is an almost DP set, so every order weakly compact operator $T$ from a Banach lattice into a Banach lattice $F$ with the weak DP property is order almost DP operator, but
an order weakly compact
operator is not necessarily order almost DP. In fact, the closed unit ball $B_{\ell_2}$ of the Banach
lattice ${\ell_2}$ is a relatively weakly compact set in ${\ell_2}$, but it is not almost DP; that is, there exist a relatively weakly compact set in ${\ell_2}$ which is not almost DP. Since $L^1[0,1]$ has order continuous norm, by \cite{Aliprantis}, every order interval in $L^1[0,1]$ is relatively weakly compact. So each operator $T: L^1[0,1] \rightarrow‎ \ell_2$ is order weakly compact, but it is not order almost DP (and so it is not order almost limited).
\end{proof}
 
 \noindent Every order almost DP operator is not an order weakly compact operator, in general. Indeed $Id_{\ell_{\infty}}: \ell_{\infty}‎\rightarrow‎ \ell_{\infty}$ is order almost DP operator and it is not order weakly compact. 
 
 \begin{theorem}\label{UU} Every operator $T$ from a Banach lattice with order continuous norm into a Banach lattice with the weak DP property is an order almost DP operator.
 \end{theorem}
 \begin{proof} Since $E$ has order continuous norm, so by \cite{Aliprantis}, for each $x\in E^+$, order interval $[-x,x]$ is relatively weakly compact. So $T[-x,x]$ is relatively weakly compact and so it is almost DP. Hence $T$ is an order almost DP operator.
 \end{proof}
 
\noindent A Banach lattice $E$ is said to be a $KB$-space, whenever every increasing norm
bounded sequence of $E^+$ is norm convergent and it is called a dual Banach lattice if $E = G^*$
for some Banach lattice $G$. A Banach lattice $E$ is called a dual $KB$-space if $E$ is
a dual Banach lattice and $E$ is a $KB$-space. It is clear that each $KB$-space has an order continuous norm. 

\begin{theorem}\label{T} Every almost weakly limited operator from a Banach space $X$ into a Banach lattice $E$ is almost limited, if one of the following assertions is valid.
\begin{itemize}
\item [(a)] The norm of the topological bidual $E^{**}$ is order continuous.
\item [(b)] $E$ is a dual $KB$-space.
\end{itemize}
\end{theorem}
 \begin{proof} If $E^{**}$  has order continuous norm or $E$ is a dual $KB$-space, then by \cite[Theorem 2.9]{Almost DP set}, every almost DP set in $E$ is $L$-weakly compact and so it is almost limited. So every almost weakly limited operator from a Banach space $X$ into a Banach lattice $E$ is almost limited.
\end{proof}

\noindent Note that there exist operators which are order almost DP and weak and almost
DP, but fail to be (almost) weakly limited. Indeed, $Id_{\ell_1}: {\ell_1}‎\rightarrow‎ {\ell_1}$ is order almost DP
and weak almost DP, but it is not (almost) weakly limited, because $\ell_{\infty}$ does not have the (positive) Schur property and so the closed unit ball $B_{\ell_1}$ is not an (almost) DP set.

\noindent In the following result, similarly to \cite[Theorem 4.2]{ordered almost limited}, we characterize Banach lattices $E$ and $F$ on
which each operator from $E$ into $F$ which is order almost DP and weak almost
DP, is almost weakly limited.\\
To establish this result, we will need the following Lemma.

\begin{lemma}\label{TT}
A Banach lattice $E$  such that $E^*$ does not have the positive Schur property if and only if there exist a disjoint weakly null sequence $(f_n)$ in $E^*_{+}$ and a sequence $(x_n)$ in $B_E^+$ and $\epsilon >0$ such that $|f_n(x_n)|>‎\epsilon‎$, for all $n$.
 \end{lemma}
  
\begin{theorem}\label{TTT} The following assertions are equivalent:
\begin{itemize}
\item [(a)] each order almost DP and weak and almost DP operator $T: E‎\rightarrow‎ F$ is
almost weakly limited;
\item [(b)] $E^*$ has order continuous norm or $F^*$ has the positive Schur property.
\end{itemize}
\end{theorem}
\begin{proof} $(a)\Rightarrow (b).$ Assume $(b)$ is false, i.e, the norm of $E^*$ is not order
continuous and $F^*$ does not have the positive Schur property. We will construct an operator $T: E‎\rightarrow‎ F$ which is order almost DP and weak and almost DP but is
not almost weakly limited. Indeed, the norm of $E^*$ is not order
continuous. By \cite[Theorem 4.2]{ordered almost limited}, we may assume that ${\ell_1}$ is a closed sublattice of $E$ and from \cite[Proposition 2.3.11]{Meyer}, it follows that there is a positive projection $P$ from
$E$ into ${\ell_1}$.\\
On the other hand, $F^*$ does not have the positive Schur property, from precedding lamma, it follows that there exist a disjoint weakly null sequence $(f_n)$ in $F^*_{+}$ and a sequence $(x_n)$ in $B_F^+$ and $\epsilon >0$ such that $|f_n(x_n)|>‎\epsilon‎$, for all $n$.
Now, we consider the operator $T=SoP :E \rightarrow {\ell_1} \rightarrow F$ where $S$ is an operator which is defined in \cite[Theorem 4.2]{ordered almost limited}. Since ${\ell_1}$ has the Schur property, then $T$ is weak
and almost limited and so it is weak and almost DP. The operator $T$ is also order almost limited and so is order almost DP.\\
But, similarly to $((1)\Rightarrow (2)$) in \cite[Theorem 4.2]{ordered almost limited}, the operator $T$ is not almost weakly limited.\\ 
$(b)\Rightarrow (a).$ Let $(f_n)$ be a disjoint weakly null sequence of $F^*$. We have to
prove that $\|T^*f_n\|\rightarrow 0$. By using Corollary 2.7 of Dodds-Fremlin \cite{Dodds}, it suffices
to prove that $|T^*f_n|(x)\rightarrow 0$ for each $x\in E^+$ and $T^*f_n(x_n) \rightarrow 0$, for every norm bounded
disjoint sequence $(x_n)‎\subset‎ E^+$. \\
Indeed, as $(f_n)$ is a disjoint weakly null sequence of $F^*$ and $T$ is order almost DP, we have $|T^*f_n|(x)\rightarrow 0$ for each $x\in E^+$. On the other hand, since the norm of $E^*$ is order continuous, by using Corollary 2.7 of Dodds-Fremlin \cite{Dodds}, it follows that $(x_n)$ is weakly null. Hence,
as $T$ is a weak almost DP operator, we obtain $T^*f_n(x_n) = f_n(T(x_n))\rightarrow 0$.
We conclude that $T$ is an almost DP operator.\\
If $F^*$ has the positive Schur property, the closed unit ball $B_F$ is almost DP set and clearly every operator $T: E‎\rightarrow‎ F$ is almost weakly limited.
\end{proof}

\noindent Note that ‎continuous linear images of DP (resp. limited) sets or sequences are DP (resp. limited), but the same conclusion is false for almost DP sets (resp. almost limited) or sequences, in general. 
  
\noindent In the following theorem, we stablish some conditions which guarantees the continuous linear images of almost limited (resp. almost DP) sets are also almost limited (resp. almost DP).\\
 Recall from \cite{Meyer} that a positive linear operator $T:E\rightarrow F$ between two Banach lattices is almost interval preserving, if $T[0,x]$ is dense in $[0,Tx]$, for every $x\in E^{+}$.  
  
 \begin{theorem}\label{W}
Let $T: E\rightarrow F$ between two Banach lattices be an almost interval preserving operator and let $A$ be an almost DP (resp. almost limited) subset of $E$. Then $T(A)$ is almost DP (resp. almost limited) in $F$.
  \end{theorem}

\begin{proof} Let $(y_n^*)$ be a disjoint weakly null (resp. disjoint weak$^*$ null) sequence in $F^*$. By \cite[Theorem 1.4.19]{Meyer}, $T^*$ is lattice homomorphism and so $(T^*y_n^*)$ is a disjoint weakly null (resp. disjoint weak$^*$ null) sequence in $E^*$. Since $A$ is almost DP (resp. almost limited), $$\limsup_{x\in A}|\langle Tx,y_n^*\rangle|= \limsup_{x\in A} |\langle x,T^*y_n^*\rangle|\rightarrow 0.$$ This completes the proof.
\end{proof}

\begin{corollary}\label{WW}
Every almost interval preserving operator $T: E\rightarrow F$ between two Banach lattices is an order almost DP operator.
  \end{corollary}
  
  \begin{proof} By \cite[Corollary 2.2]{Almost DP set}, for each $x\in E^+$, order interval $[-x,x]$ is an almost DP set, and by Theorem \ref{W}, $T[-x,x]$ is an almost DP set.
\end{proof}

\begin{theorem}\label{E}
Every almost interval preserving operator $T$ on a Banach lattice $E$ order isomorphic to some $C(K)$ space is an order almost weakly limited operator.
  \end{theorem}
  \begin{proof}
 If $E$ is order isomorphic to some $C(K)$ space, then every norm bounded set in $E$ is order bounded and by \cite[Corollary 2.2]{Almost DP set}, it is almost DP. So by Theorem \ref{W}, the operator $T$ is almost weakly limited. 
  \end{proof}
  
\noindent By the tecnique in the proof of Theorem 2.7 in \cite{almost limited set}, we have the following theorem.

\begin{theorem}\label{Q}
Let $E$ and $F$ be two Banach lattices such that $E^*$ has the weakly sequentially continuous lattice operations. If $T:E\rightarrow F$ is an operator, then $T(A)$ is an almost DP set in $F$, whenever $A$ is an almost DP solid set in $E$.
  \end{theorem}

\begin{proof} Suppose that $(f_n)$ is a disjoint weakly null sequence in $F^*.$ We claim that $(|T^*(f_n)|)$ is weakly null in $E^*.$ In fact, if $E^*$ has the weakly sequentially continuous lattice operations then $(|T^*(f_n)|)$ is weakly  null, since $(T^*f_n)$ is weakly null in $E^*$.\\
Now, to finish the proof, we have to show that $$\sup_{y\in TA}|f_n(y)|= \sup_{x\in A}|T^*f_n(x)|\rightarrow 0.$$ Let $(x_n)\subseteq A^{+}$ be a disjoint sequence. Since the sequence $(|T^*(f_n)|)$ is weakly null in $E^*$, by \cite[page 77]{Aliprantis}, there exists a disjoint sequence $(g_n)\subseteq E^*$ such that $|g_n|\leq |T^*f_n|$ and $g_n(x_n)=(T^*f_n)(x_n)$, for all $n.$
It is easy to see that the sequence $(g_n)$ is weakly null. As the set $A$ is almost DP, $‎\sup_{x\in A}|g_n(x)|\rightarrow 0.$ From the inequality$$|T^*f_n(x_n)|=|g_n(x_n)|\leq \sup_{x\in A}|g_n(x)|,$$ we conclude that $T^*f_n(x_n)\rightarrow 0.$
Now, by \cite[Theorem 2.4]{Dodds}, we have $$\sup_{x\in A}|T^*f_n|(|x|)=\sup_{x\in A}|T^*f_n|(x)\rightarrow 0.$$ Thus by the inequality $|T^*f_n(x)|\leq |T^*f_n|(|x|),$ we see that $\sup_{x\in A}|T^*f_n(x)|\rightarrow 0$ and so $T(A)$ is almost DP.
\end{proof}

\noindent According to \cite[Lemma 2.2]{almost limited set}, for almost limited sets, we have the following lemma for almost DP sets.
\begin{lemma}\label{QQ} Let $A$ be a norm bounded subset of a Banach lattice $E$. If for every $‎\epsilon‎ >0$ there exists some almost DP subset $A_‎\epsilon$ of $E$ such that
$A‎\subset‎ A_‎{\epsilon} + \epsilon B_E$, then $A $ is almost DP.
\end{lemma}

 \begin{theorem}\label{Z} Let $E$, $F$ and $G$ be three Banach lattices. Then
\begin{itemize}
\item [(a)] the class of order almost DP operators is a norm closed vector subspace
of the space $L(E, F)$ of all operators from $E$ into $F$.
\item [(b)] if $T: E‎\rightarrow‎ F$ is an order almost DP operator, then for each almost interval preserving
operator $S: F‎\rightarrow‎ G$, the composed operator $S o T$ is order almost DP.
\item [(c)] if $T: E‎\rightarrow‎ F$ is an order bounded operator, then for each order almost DP
operator $S: F‎\rightarrow‎ G$, the composed operator $S o T$ is order almost DP.
\end{itemize}
\end{theorem}
 \begin{proof} $(a).$ Clearly the class of order almost DP operators is a vector subspace of $L(E, F)$ and by Lemma \ref{QQ}, this class is also norm closed.\\
 $(b).$ Let $T: E‎\rightarrow‎ F$ be an order almost DP operator. Then for each $x\in E^+$,
$T([-x, x])$ is an almost DP. Since $S$ is almost interval preserving, $S(T([-x,x]))$
is an almost DP set in $F$, and hence $S o T$ is order almost DP.\\
 $(c).$ Let $T: E‎\rightarrow‎ F$ be an order bounded operator. Then for each $x\in E^+$,
$T([-x, x])$ is an an order bounded. Since $S$ is order almost DP, $S(T([-x,x]))$
is an almost DP set in $F$, and hence $S o T$ is order almost DP.
 \end{proof}
 
 \begin{theorem}\label{ZZ} Let $T$ be an operator from a Banach lattice $E$ into a Banach
lattice  $F$. If $T^*$ is almost DP, then $T$ is order almost DP.
 \end{theorem}
 \begin{proof} Let $(f_n)$ be a disjoint weakly null sequence of $F^*$.
As the adjoint $T^*$ is almost DP, we deduce that $\|T^*f_n\|\rightarrow 0$. So for each $x\in E^+$ order interval $T[-x,x]$ is an almost DP; that is, $$\sup_{z\in T[-x,x]}|f_n(z)|=\sup_{y\in [-x,x]}|T^*f_n(y)| =|T^*f_n|(x)\rightarrow 0.$$ We deduce that $T$ is order Dunford-Pettis.
 \end{proof}
 
 \begin{theorem}\label{XX} Let $T$ be an operator from a Banach lattice $E$ into a Banach
lattice  $F$. If $T^*$ is almost DP, then $T$ is weak and almost DP.
 \end{theorem}
 \begin{proof} Let $(x_n)$ be a weakly null sequence of $E$ and $(f_n)$ be a disjoint weakly null sequence in $F^*$. We have to prove that $f_n(T (x_n))‎\rightarrow‎0$. As $(f_n)$ is a disjoint weakly null sequence in $F^*$ and hence $T^*$ is almost DP then $\|T^*f_n\|\rightarrow 0$.
On the other hand, since $(x_n)$ is a weakly null sequence of $E$, hence $(x_n)$ is
norm bounded and by the inequality $|T^*(f_n)(x_n))‎| =|f_n(T (x_n))‎| ‎\leq ‎\|T^*f_n\|$ , we conclude that $T$
is weak and almost DP.
 \end{proof}

\noindent Similarly to  \cite[Theorem 3.1]{ordered DP}, we have the following theorem.
\begin{theorem}\label{2.10} The following assertions are equivalent:
\begin{itemize}
\item [(a)] Each order almost DP and weak and almost DP operator $T: E‎\rightarrow‎ F$ has an adjoint almost DP operator;
\item [(b)] $E^*$ has order continuous norm or $F^*$ has the positive Schur property.
\end{itemize}
\end{theorem}

\bibliographystyle{amsplain}

\begin{thebibliography}{99}

\bibitem{Aliprantis} C. D. Aliprantis and O. Burkishaw, \textit{Positive operators}, Academic Press, New York and London, 1978.

\bibitem{Andrews} K. T. Andrews, \textit{Dunford-Pettis sets in the space of Bochner integrable functions},
Math. Ann. \textbf{241}, (1979), 35--41. 

\bibitem{Dunford-Pettis sets} B. Aqzzouz and K. Bouras, \textit{Dunford-Pettis sets in Banach lattices}, Acta Math. Univ. Comenianae, \textbf{2} (2012), 185--196.

\bibitem{weak DP property} B.Aqzzouz and M. Moussa, \textit{Banach lattices with weak Dunford-Pettis property}, World Academy of Science, Engineering and Technology, \textbf{5}, (2011) ,02--23.

\bibitem{Limited} J. Bourgain and J. Diestel, \textit{Limited operators and strict cosingularity}, Math. Nachr. 119 (1984), 55--58.

\bibitem{ordered DP} K.Bouras, E. Kaddouri, J. Hmichane and M. Moussa, \textit{The class of ordered Dunford-Pettis operators }, Mathematica Bohemica, \textbf{138} (2013), 289--297.

\bibitem{Almost DP set} K.Bouras, \textit{Almost Dunford-Pettis sets in Banach lattices}, Rend. Circ. Mat. Palermo \textbf{62} (2013), 227--236.

\bibitem{Almost limited set} J. X. Chen, Z.L.Chen and G.X.Ji, \textit{Almost limited sets in Banach lattices}, J. Math. Anal. Appl. \textbf{412}  (2014), 547--563.

\bibitem{Dodds}  P. G. Dodds and D. H. Fremlin, \textit{Compact operators on Banach lattices}, Israel J. Math. \textbf{34} (1979), 287--320.

\bibitem{AZIZ ELBOUR} A. Elbour, N. Machrafi and M. Moussa, \textit{On the class of weak and almost limited operators}, http://arxiv.org/abs/1403.0136.

\bibitem{ordered limited} A. E. Kaddouri, M. Moussa, \textit{About the class of ordered limited operators}, Acta Universitatis Carolinae. Mathematica et Physica, \textbf{54} (2013), 37--43.

\bibitem{ordered almost limited} A. E. Kaddouri, K. EL. Fahri, J. Hmichane and M. Moussa, \textit{The class of ordered almost limited operators on Banach lattices}, submitted to Acta Math. Univ. Comenianae.

\bibitem{weakly limited} H. Z. Lin, \textit{The weakness of limited set and limited operator in Banach spaces}, Pure Appl. Math. (Xi’an),
\textbf{27}, (2011), 650-655 (in Chinese).

\bibitem{Meyer} P. Meyer- Nieberg, \textit{Banach lattices},  Universitext, springer- Verlag, Berlin, 1991.

\bibitem{almost limited set} N. Machrafi, A. Elbour and  M. Moussa, \textit{Some characterizations of almost limited sets and applications}, http://arxiv.org/abs/1312.2770, 10 Dec 20.

\bibitem{Sanchez almost DP operator} J.A. Sanchez, \textit{Operators on Banach lattices} (Spanish), Ph. D. Thesis, Complutense University, Madrid, 1985.

\bibitem{positive Schur} J. A. Sanchez, \textit{Positive Schur property in Banach lattices}, Extraccta Mathematica, \textbf{7}, (1992), 161-163.

\bibitem{Weak DP Property Wnuk} W. Wnuk, \textit{Banach Lattices with the weak Dunford--Pettis Property}, Atti Sem.Mat. Univ. Modena XLII, (1994), 227--236.

\end{thebibliography}
 
\end{document}